\documentclass[11 pt]{amsart}
\usepackage{graphicx}
\usepackage[hidelinks]{hyperref}
\usepackage{color}
\usepackage{amssymb,amsmath,amsfonts,latexsym}
\usepackage{epstopdf}
\usepackage{nicefrac}
\usepackage{mathtools}
\usepackage{amsthm}
\usepackage{hyperref}
\hypersetup{
     colorlinks   = true,
     citecolor    = blue
}

\usepackage{tikz}
\usepackage{tikz,fullpage}
\usetikzlibrary{patterns}
\usepackage{enumerate}
\numberwithin{equation}{section}
\usepackage[margin=1.27in]{geometry}
\usetikzlibrary{arrows,%
                petri,%
                topaths}%
\usepackage{tkz-berge}
\usepackage[position=top]{subfig}
\usetikzlibrary{%
  matrix,%
  calc,%
  arrows%
}
\newtheorem{thm}{THEOREM}[section]
\newtheorem{conj}[thm]{CONJECTURE}

\newtheorem{defn}[thm]{DEFINITION}

\newtheorem{lemma}[thm]{LEMMA}

\begin{document}
\title{Cohomology of $\mathcal A_\theta^{alg} \rtimes \mathbb Z_2$ and its Chern-Connes Pairing}
\author{Safdar Quddus}

\date{\today}
 
\let\thefootnote\relax

\maketitle

\begin{abstract}
We calculate the Hochschild and cyclic cohomology of the noncommutative $\mathbb Z_2$ toroidal algebraic orbifold $\mathcal A_\theta^{alg} \rtimes \mathbb Z_2$. We also calculate the  Chern-Connes pairing  of the even periodic cyclic cocycles with the known elements of $K_0(\mathcal A_\theta^{alg} \rtimes \mathbb Z_2)$.
\end{abstract}

\section{Introduction and Statement}
In the classic paper \cite{C}, Connes constructed a noncommutative analogue of the Chern map of differential geometry. He considered the map from the $K_0$ group of a noncommutative algebra to its even cyclic homology and paired a projection with the cyclic cocycle to give a numerical invariant. \par
Let $\mathcal S(\mathbb Z^2)$ be the Schwartz space on $\mathbb Z^2$, consisting of all complex sequences $a_{n,m}$ satisfying:
$$\displaystyle \sup_{(n,m) \in \mathbb Z^2}(|n| + |m|)^q |a_{n,m}| < \infty, \text{  } \text{   for all } q \in \mathbb N .$$
For given $\theta \in \mathbb R$, we associate the algebra $\mathcal A_\theta$ defined below.  
$$\mathcal A_\theta =\left\{a=\displaystyle\sum_{(n,m)\in \mathbb Z^2} a_{n,m} U_1^n U_2^m \mid a_{n,m} \in \mathcal S(\mathbb Z^2) \right\},$$
where $U_1$ and $U_2$ are unitary generators satisfying $U_2 U_1 = \lambda U_1 U_2$, $\lambda =e^{2 \pi i \theta}$.
Connes \cite{C} computed the cyclic cohomology and Chern-Connes index for the smooth algebra $\mathcal A_\theta$.  The group $SL(2,\mathbb Z)$ has the following action on $\mathcal A_\theta$. An element
$$g= \left[
 \begin{array}{cc}
   g_{1,1} & g_{1,2} \\
   g_{2,1} & g_{2,2}
 \end{array} \right]\in SL(2,\mathbb Z)$$
acts on the generators $U_1$ and $U_2$ as described below:
$$g \cdot U_1=e^{(\pi i g_{1,1} g_{2,1})\theta}U_1^{g_{1,1}}U_2^{g_{2,1}} \text{ and } g \cdot U_2=e^{(\pi i g_{1,2} g_{2,2})\theta}U_1^{g_{1,2}}U_2^{g_{2,2}}.$$
 Let $\mathcal A_\theta^{alg}$ consists  of all finitely supported elements of $\mathcal A_\theta$. We shall study the crossed product of the subalgebra $\mathcal A_\theta^{alg}$ with the group $\mathbb Z_2$ identified as a subgroup of $SL(2,\mathbb Z)$.

Berest et al. \cite{BRT} calculated the Picard group of $\mathcal A_\theta^{alg}$. Some of the Hochschild homology groups of $\mathcal A_\theta^{alg} \rtimes \mathbb Z_k$ for $k=2,3,4$ and $6$ were known for many years (\cite{O} and \cite{B}). All the Hochschild and cylic homology groups of the orbifolds $\mathcal A_\theta^{alg} \rtimes \mathbb Z_k$, for $k=2,3,4$ and $6$, were recently calculated \cite{Q}. Further the Hochschild homology of the Weyl algebra was studied by Alev and Lambre in \cite{AL}. In this article we shall compute the Hochschild and cyclic cohomology of $\mathcal A_\theta^{alg} \rtimes \mathbb Z_2$,  the $\mathbb Z_2$ noncommutative algebraic toroidal orbifold. We also compute the Chern-Connes index for this orbifold by pairing these cocycles with the algebraic projections of the group $K_0(\mathcal A_\theta \rtimes \mathbb Z_2)$, which was calculated in \cite{ELPH}. In this article we adopt the notation from \cite{C} and \cite{Q} and prove the following results.

\begin{thm} \label{thm:tell} If $\theta \notin \mathbb Q$, then the Hochschild cohomology groups of $\mathcal A_\theta^{alg} \rtimes \mathbb Z_2$ are:
\begin{enumerate}
\item[ ] $H^0(\mathcal A_{\theta}^{alg} \rtimes \mathbb Z_2,(\mathcal A_{\theta}^{alg} \rtimes \mathbb Z_2)^\ast ) \cong \mathbb C^{5}$,
\item[ ] $H^1(\mathcal A_{\theta}^{alg} \rtimes \mathbb Z_2,(\mathcal A_{\theta}^{alg} \rtimes \mathbb Z_2)^\ast) = 0$, and
\item[ ] $H^2(\mathcal A_{\theta}^{alg} \rtimes \mathbb Z_2,(\mathcal A_{\theta}^{alg} \rtimes \mathbb Z_2)^\ast)  \cong \mathbb C$.
\end{enumerate}
\end{thm}

\begin{thm} \label{thm:cneat}
$HP^{even}(\mathcal A_\theta^{alg} \rtimes \mathbb Z_2) \cong \mathbb C^6$ and $HP^{odd}(\mathcal A_\theta^{alg} \rtimes \mathbb Z_2) = 0$.
\end{thm}

\begin{thm} \label{thm:treat}
The following is the description of the Chern-Connes pairing of the six dimensional group $HP^{even}(\mathcal A_\theta^{alg} \rtimes \mathbb Z_2)$ generated by cocycles  S$\tau$, $S\mathcal D_{1,1}$, $S\mathcal D_{0,0}$, $S\mathcal D_{0,1}$, $S\mathcal D_{1,0}$ and $\varphi$, with the five known independent projections of $\mathcal A_\theta^{alg} \rtimes \mathbb Z_2$ namely 1, $p^\theta$, $q_1^\theta$, $q_2^\theta$ and $r^\theta$ \cite{ELPH}
\begin{center}
  \begin{tabular}{c || c | c | c | c | c | c }
     &  S$\tau$ & $S\mathcal D_{1,1}$ & $S\mathcal D_{0,0}$ & $S\mathcal D_{0,1}$ & $S\mathcal D_{1,0}$ & $S\varphi$  \\ \hline \hline
    1 & $1$ & $0$ & $0$ & $0$ & $0$ & $0$\\ \hline
    $p^\theta$ & $\frac{1}{2}$ & $\frac{1}{2}$ & $0$ & $0$ & $0$ & $0$\\ \hline
    $q_1^\theta$ & $\frac{1}{2}$ & $0$ & $0$ & -$\frac{1}{2}$ & $0$ & $0$\\ \hline
   $q_2^\theta$ & $\frac{1}{2}$ & $0$ & -$\frac{1}{2}$ & $0$ & $0$ & $0$\\ \hline
    $r^\theta$ & $\frac{1}{2}$ & $0$ & $0$ & $0$ & -$\frac{\lambda}{2}$ & $0$\\
    \hline
  \end{tabular}.
\end{center}
\end{thm}
We end the article with a conjecture over the dimension of the unknown group $K_0(\mathcal A_\theta^{alg} \rtimes \mathbb Z_2)$.

\section{Hochschild cohomology of $\mathcal A_\theta^{alg} \rtimes \mathbb Z_2$}
We note that the dual of the algebraic noncommutative torus $\mathcal A_\theta^{alg}$ is 
$$\mathcal A_\theta^{alg \ast} =\left\{a \mid a =\displaystyle\sum_{(n,m)\in \mathbb Z^2} a_{n,m} U_1^n U_2^m\right\},$$
where $U_1$ and $U_2$ are unitaries satisfying $U_2 U_1 = \lambda U_1 U_2$. For $a \in \mathcal A_\theta^{alg}$, let the trace $\tau$ on the algebra $\mathcal A_\theta^{alg}$ be defined as
$$\tau(a) = a_{0,0}.$$
Then an element $a \in \mathcal A_\theta^{alg \ast}$ acts on $b \in \mathcal A_\theta^{alg}$ as $a(b) = \tau(ab)$. Using the results of Getzler and John \cite{GJ}, the cohomology group $H^\bullet(\mathcal A_\theta^{alg} \rtimes \mathbb Z_2, (\mathcal A_\theta^{alg} \rtimes \mathbb Z_2)^\ast$) has the following decomposition:
\begin{center}
 $H^\bullet(\mathcal A_\theta^{alg} \rtimes \mathbb Z_2, (\mathcal A_\theta^{alg} \rtimes \mathbb Z_2)^\ast) =\displaystyle \bigoplus_{g \in \mathbb Z_2} H^{\bullet}(\mathcal A_\theta^{alg}, {}_{g}\mathcal A_\theta^{alg \ast})^{\mathbb Z_2}=  H^\bullet(\mathcal A_\theta^{alg}, \mathcal A_\theta^{alg \ast})^{\mathbb Z_2} \displaystyle \bigoplus H^\bullet(\mathcal A_\theta^{alg}, {}_{-1}\mathcal A_\theta^{alg \ast})^{\mathbb Z_2}$.
\end{center}
 In the above equation, ${}_{-1}\mathcal A_\theta^{alg \ast}$ consists of elements of $\mathcal A_\theta^{alg \ast}$ with the following twisted $\mathcal A_\theta^{alg}$ bimodule structure. For $a \in {}_{-1}\mathcal A_\theta^{alg \ast}$ and $ \alpha \in \mathcal A_\theta^{alg}$,
$$\alpha \cdot a = (-1 \cdot \alpha)a\text{ and } a \cdot \alpha = a\alpha,$$
where $a \alpha$ is the product of $a$ and $\alpha$ in $\mathcal A_\theta^{alg \ast}$.
We recall the modified Connes projective resolution:
$$\mathcal A^{alg}_\theta \xleftarrow{\epsilon} \mathcal B^{alg}_\theta \xleftarrow{b_1} \mathcal B^{alg}_\theta \displaystyle \bigoplus \mathcal B^{alg}_\theta \xleftarrow{b_2} \mathcal B^{alg }_\theta$$
where $$\mathcal B^{alg}_\theta = \mathcal A_\theta^{alg} \otimes (\mathcal A_\theta^{alg})^{op},$$
$$\epsilon(a\otimes b ) = ab,$$ 
$$b_1(1\otimes e_j)= 1\otimes {U_j}- {U_j}\otimes 1,$$ 
$$b_2(1\otimes( e_1 \wedge e_2 ) ) = (U_2\otimes 1 - \lambda \otimes U_2 )\otimes e_1- ( \lambda U_1\otimes 1 - 1\otimes U_1 )\otimes e_2.$$
The above resolution was used in \cite{Q} to calculate the Hochschild and cyclic homology groups of the algebra $\mathcal A_\theta^{alg} \rtimes \mathbb Z_k$ for $k=2,3,4$ and $6$. We use it to construct the twisted cochain complex corresponding to each of the two elelments of the group $\mathbb Z_2$. Thereafter we compute the cohomology groups of $\mathcal A_\theta^{alg} \rtimes \mathbb Z_2$ by locating the $\mathbb Z_2$ invariant cocycles.\par
While the bar resolution is not computationally convenient, modified Connes resolution for the algebraic noncommutative torus does make it easier to compute the cohomology groups. In order to locate the $\mathbb Z_2$ invariant cocycles of $H^\bullet(\mathcal A_\theta^{alg}, {}_{g}\mathcal A_\theta^{alg \ast})$, we need to use the resolution homotopy maps
$$h_{\ast}: C_\ast(\mathcal A_\theta^{alg}) \to J_\ast(\mathcal A_\theta^{alg})$$ and $$k_{\ast}: J_\ast(\mathcal A_\theta^{alg}) \to C_\ast(\mathcal A_\theta^{alg}),$$
where $J_\ast(\mathcal A_\theta^{alg})$ is the standard bar resolution $(J_k(\mathcal A_\theta^{alg}) = \mathcal B_\theta^{alg} \otimes (\mathcal A_\theta^{alg})^{\otimes k})$ and $C_\ast(\mathcal A_\theta^{alg})$ is the Connes resolution. We push a cocylcle $\mathcal D$ into the bar complex and let $\mathbb Z_2$ act on it. Then, in the Connes complex, we compare the pullback of this $\mathbb Z_2$-acted cocycle with $\mathcal D$ to check the $\mathbb Z_2$ invariance.
These maps were explicitly calculated in \cite{C} and \cite{Q}.

It is worthwhile to note that $\text{Hom}_{\mathcal B_\theta^{alg}}(\mathcal B_\theta^{alg}, {}_{-1}\mathcal A_\theta^{alg \ast})$ and  $\text{Hom}_{\mathcal B_\theta^{alg}}(\mathcal B_\theta^{alg}, {}_{}\mathcal A_\theta^{alg \ast})$ can be identified with ${}_{-1}\mathcal A_\theta^{alg \ast}$ and ${}_{}\mathcal A_\theta^{alg \ast}$, respectively.
Hence for $g=-1$ we have the following Hochschild cohomology complex:
$${}_{-1}\mathcal A_\theta^{alg \ast} \xrightarrow{{}_{-1}\alpha_1}{}_{-1}\mathcal A_\theta^{alg \ast} \oplus {}_{-1}\mathcal A_\theta^{alg \ast} \xrightarrow{{}_{-1}\alpha_2} {}_{-1}\mathcal A_\theta^{alg \ast} \rightarrow 0$$
where for $\varphi, \varphi_1 \text{ and } \varphi_2$ in $\mathcal A_\theta^{alg \ast}$, the maps  ${}_{-1}\alpha_1$ and  ${}_{-1}\alpha_2$ are as follows:
$${}_{-1}\alpha_1(\varphi) = (U_1^{-1}\varphi -\varphi U_1, U_2^{-1}\varphi - \varphi U_2);$$
$${}_{-1}\alpha_2(\varphi_1,\varphi_2)=U_2^{-1}\varphi_1 - \lambda \varphi_1 U_2 - \lambda U_1^{-1}\varphi_2 + \varphi_2U_1.$$

\begin{lemma}
$H^0(\mathcal A_\theta^{alg},{}_{-1}\mathcal A_\theta^{alg\ast})^{\mathbb Z_2} \cong \mathbb C^4$.
\end{lemma}

\begin{proof} Let $\varphi = \sum \varphi_{n,m}U_1^nU_2^m$ be an element of ${}_{-1}\mathcal A_\theta^{alg \ast}$. Then $\varphi$ is a 0-cocycle if and only if ${}_{-1}\alpha_1(\varphi) =0$, which implies that
$ U_1^{-1}\varphi -\varphi U_1=U_2^{-1}\varphi - \varphi U_2=0$. This further gives the relation $\varphi_{n+1,m} = \lambda^m \varphi_{n-1,m} = \lambda^{m+n-1} \varphi_{n-1,m-2}$ on its coefficients. Hence we see that $$H^0(\mathcal A_\theta^{alg},{}_{-1}\mathcal A_\theta^{alg\ast}) \cong \mathbb C^4.$$ The generators of this group are the cocycles generated by $\varphi_{0,0}$, $\varphi_{0,1}$, $\varphi_{1,0}$ and $\varphi_{1,1}$. \par
Let us denote by $\mathcal D_{i,j}$ the cocycle generated by $\varphi_{i,j}$, for $0 \leq i,j \leq 1$.
First consider the cocycle $\mathcal D_{0,0}$. The above relation on the coefficients of $\mathcal D_{0,0}$ gives $\varphi_{2n,2m} = {\lambda}^{2mn}\varphi_{0,0}$ for all $(2n,2m) \in \mathbb Z^2$. The maps $k_0$ and $h_0$ are idenity and hence the action of $\mathbb Z_2$ on $H^0(\mathcal A_\theta^{alg},{}_{-1}\mathcal A_\theta^{alg\ast})$ is given by $U_j  \to U_j^{-1}$ for $j=1,2$. Thus we conclude that the $\mathbb Z_2$ action leaves $\mathcal D_{0,0}$ invariant. \par

For the cocycle $\mathcal D_{0,1}$, we infer that $\varphi_{2k,2l+1} = \lambda^{2kl+k} \varphi_{0,1}$.  Now from
$$\varphi_{-2k,-2l-1} = \varphi_{2(-k),2(-l-1)+1}=\lambda^{2(-k)(-l-1)+(-k)} \varphi_{0,1}=\lambda^{2kl+k} \varphi_{0,1}$$
it follows that $\mathcal D_{0,1}$ is a $\mathbb Z_2$ invariant element of $H^0(\mathcal A_\theta^{alg},{}_{-1}\mathcal A_\theta^{alg \ast})$. In the case of $\mathcal D_{1,0}$, its coeffcients satisfy the relation $\varphi_{2k+1,2l} = \lambda^{2kl+l} \varphi_{1,0}$. Since $$\varphi_{-2k-1,-2l} = \varphi_{2(-k-1)+1,2(-l)}=\lambda^{2(-k-1)(-l)+(-l)} \varphi_{1,0}=\lambda^{2kl+l} \varphi_{1,0},$$ the cocycle $\mathcal D_{1,0} \in H^0(\mathcal A_\theta^{alg},{}_{-1}\mathcal A_\theta^{alg\ast})^{\mathbb Z_2}$. Finally for $\mathcal D_{1,1}$, we have $\varphi_{2k+1,2l+1} = \lambda^{2kl+k+l} \varphi_{1,1}$ and 
$\varphi_{-2k-1,-2l-1} = \lambda^{2kl+k+l} \varphi_{1,1}$. Hence $\mathcal D_{1,1} \in H^0(\mathcal A_\theta^{alg},{}_{-1}\mathcal A_\theta^{alg\ast})^{\mathbb Z_2}.$
\end{proof}

\begin{lemma}
$H^2(\mathcal A_\theta^{alg},{}_{-1}\mathcal A_\theta^{alg\ast})^{\mathbb Z_2} =0.$
\end{lemma}
\begin{proof}
Let $\varphi \in \mathcal A_\theta^{alg \ast}$ and let $\widetilde{\varphi}$ be the corresponding element of $\text{Hom}{\mathcal B_\theta^{alg}}(J_2,\mathcal A_\theta^{alg\ast})$. Then
$$\widetilde{\varphi}(a\otimes b \otimes e_1 \wedge e_2)(x)=\varphi((-1\cdot b)xa),$$
for all  $a,b,x \in \mathcal A_\theta^{alg}$.
Let $\psi = k_2^{\ast} \widetilde{\varphi} = \widetilde{\varphi} \circ k_2$. We have
$$\psi(x_0,x_1,x_2)=\widetilde{\varphi}(k_2(I \otimes x_1 \otimes x_2))(x_0),$$
for all $x_0,x_1,x_2 \in \mathcal A_\theta^{alg}$.
The group $\mathbb Z_2$ acts on $\mathcal A_\theta^{alg}$ in the bar complex as $$-1 \cdot \chi(x_0,x_1,x_2)=\chi(-1\cdot x_0,-1 \cdot x_1, -1 \cdot x_2).$$ 
Further we pull the map ${}_{-1}\psi = -1 \cdot \psi$ back on to the Connes complex via the map $h_2^{\ast}$. Let $w=h_2^{\ast}({}_{-1}\psi)$ denote the pullback of ${}_{-1}\psi$ on the Connes complex. We have  
\begin{center}
 $w(x_0)={}_{-1}\psi(x_0,U_2,U_1)-\lambda{}_{-1}\psi(x_0,U_1,U_2)=\psi(-1 \cdot x_0, U_2^{-1},U_1^{-1}) -\lambda \psi(-1 \cdot x_0, U_1^{-1},U_2^{-1}) \newline= \widetilde{\varphi}(k_2(I \otimes U_2^{-1}\otimes U_1^{-1}))(-1\cdot x_0)-\lambda \widetilde{\varphi}(k_2(I \otimes U_1^{-1}\otimes U_2^{-1}))(-1\cdot x_0)$.
\end{center}
Following the calculations from \cite[Section 6]{Q}, we have
$$k_2(I \otimes U_2^{-1}\otimes U_1^{-1})-\lambda k_2(I \otimes U_1^{-1}\otimes U_2^{-1})= (U_1^{-1}U_2^{-1} \otimes U_1^{-1}U_2^{-1}).$$
Applying this we conclude that
\begin{center}
$ \widetilde{\varphi}(k_2(I \otimes U_2^{-1}\otimes U_1^{-1}))(-1\cdot x_0)-\lambda \widetilde{\varphi}(k_2(I \otimes U_1^{-1}\otimes U_2^{-1}))(-1\cdot x_0)=\widetilde{\varphi}(U_1^{-1}U_2^{-1} \otimes U_1^{-1}U_2^{-1})(-1\cdot x_0) = \varphi(U_1U_2\cdot (-1 \cdot x_0)\cdot U_1^{-1}U_2^{-1})$.
\end{center}
Hence we need to compare $\varphi(x)$ with $\varphi(U_1U_2\cdot (-1 \cdot x)\cdot U_1^{-1}U_2^{-1})$. Using the Connes complex, we see that $H^2(\mathcal A_\theta^{alg},{}_{-1}\mathcal A_\theta^{alg\ast})^{\mathbb Z_2} = {}_{-1}\mathcal A_\theta^{alg \ast} / Im({}_{-1}\alpha_2)$. Since ${}_{-1}\alpha_2(U_2,0) = 1-\lambda U_2^2$ and ${}_{-1}\alpha_2(0,U_1) =  U_1^2-\lambda$, we have $H^2(\mathcal A_\theta^{alg},{}_{-1}\mathcal A_\theta^{alg\ast}) \cong \mathbb C^4$ generated by the cocycles supported at $\varphi_{0,0}$, $\varphi_{1,0}$, $\varphi_{0.1}$ and $\varphi_{1,1}$.

\underline{Case 1}:
We check the invariance of $\varphi_{0,0}$. From $\varphi_{0,0}(U_1U_2\cdot (-1 \cdot x)\cdot U_1^{-1}U_2^{-1})={\lambda}^{-1}x_{0,0}$ and $\varphi_{0,0}(x) = x_{0,0}$, we see that $\varphi_{0,0}$ is \emph{not} invariant under the $\mathbb Z_2$ action.\newline

\underline{Case 2}: Observe that for $\varphi_{1,0}$, we have $\varphi_{1,0}(U_1U_2\cdot (-1 \cdot x)\cdot U_1^{-1}U_2^{-1})=x_{-1,0}$ and $\varphi_{1,0}(x) = x_{1,0}$.
Since the cocycle class $\varphi_{1,0}$ is equivalent to the class $\lambda \varphi_{-1,0}$,  it is \emph{not} invariant under the $\mathbb Z_2$ action.\newline

\underline{Case 3}: We check the invariance of $\varphi_{0,1}$. We have $\varphi_{0,1}(U_1U_2\cdot (-1 \cdot x)\cdot U_1^{-1}U_2^{-1})=x_{0,-1}$ and $\varphi_{0,1}(x) = x_{0,1}$. Since the cocycle class $\varphi_{0,1}$ is equivalent to the class $\lambda^{-1} \varphi_{0,-1}$, it is \emph{not} invariant under the $\mathbb Z_2$ action.\newline

\underline{Case 4}: Finally, we check the invariance of $\varphi_{1,1}$. We have $\varphi_{1,1}(U_1U_2\cdot (-1 \cdot x)\cdot U_1^{-1}U_2^{-1})={\lambda}^{-1}x_{-1,-1}$ and $\varphi_{1,1}(x) = x_{1,1}$. Since the cocycle class $\varphi_{1,1}$ is equivalent to the cocycle class $\varphi_{-1,-1}$, the cocycle is \emph{not} invariant under the $\mathbb Z_2$ action.

For $\psi = a \varphi_{0,0}+ b \varphi_{1,0} + c\varphi_{0,1} + d\varphi_{1,1}$, if $\Psi$ is the pullback of the corresponding cocycle in the bar complex after the $\mathbb Z_2$ action, then:
$$\Psi= a{\lambda}^{-1} \varphi_{0,0} + b {\lambda}^{-1} \varphi_{1,0} + c {\lambda}^{-1} \varphi_{0,1} + d {\lambda}^{-1} \varphi_{1,1}.$$
We see that the coefficients of this pullback are different from those of the original cocycle.
Therefore we conclude that
$$H^2(\mathcal A_\theta^{alg},{}_{-1}\mathcal A_\theta^{alg\ast})^{\mathbb Z_2} =0.$$
We remark in this computation that although $H^2(\mathcal A_\theta^{alg}, {}_{-1}\mathcal A_\theta^{alg \ast})$ is of $4$ dimension, there is no nontrivial $\mathbb Z_2$ invariant cocycle.
\end{proof}
 For $\varphi \in \mathcal A_\theta^{alg \ast} \oplus \mathcal A_\theta^{alg \ast}$, we define the diagram $Dgm(\varphi)\;( \subset \mathbb Z^2 \oplus \mathbb Z^2 \oplus \mathbb Z^2$) associated to it \cite[Section 7]{Q}. Two elements $a ,b \in \mathbb C$ indexed by the lattice $\mathbb Z^2$ are said to be $f$-connected and drawn on the lattice plane as\\
$$
\begin{tikzpicture}
\draw (-0.2,0) -- (1.2,0) node(xline)[right] {$b$};
\draw (1.2,0) -- (-0.2,0) node(xline)[left] {$a$};
\fill (canvas cs:x=-0.2cm,y=0cm) circle(2pt);
\fill (canvas cs:x=1.2cm,y=0cm) circle (2pt);
\end{tikzpicture}
$$
if there exists $f \in \mathbb C[x,y,w,z]$ whose roots are $a$ and $b$.
For example, consider the following equation
$$a^1_{6,1}-a^2_{5,0}={\lambda}a^1_{4,1}-{\lambda}^{-5}a^2_{5,2}.$$
The corresponding diagram is as below, where the boxes represent elements of $a^1_{\bullet, \bullet}$ and the thick dots that of $a^2_{\bullet, \bullet}$.
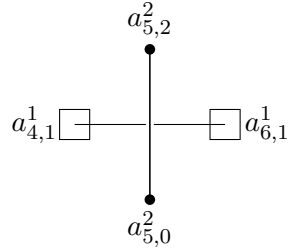
\begin{center}
\begin{tikzpicture}
\draw (-1,0)node(xline)[left] {$a^1_{4,1\text{ }}$} -- (1,0 )node(xline)[right] {$\text{  }a^1_{6,1}$};
\draw[draw=white,double=black,very thick] (0,-1)node(yline)[below] {$a^2_{5,0}$} -- (0,1)node(yline)[above] {$a^2_{5,2}$};
\fill (canvas cs:x=0cm,y=1cm) circle (2pt);
\fill (canvas cs:x=0cm,y=-1cm) circle (2pt);
\node at (-1,0) [transition]{};
\node at (1,0) [transition]{};
\end{tikzpicture}
\end{center}
\par In the above example we see that $f(x,y,w,z) = x-y + \lambda^{-5}z- \lambda w$ has its roots as $a^1_{4,1}, a^1_{6,1}, a^2_{5,0}$ and $a^2_{5,2}$. For $\varphi=(\varphi^1, \varphi^2) \in \mathcal A_\theta^{alg \ast} \oplus \mathcal A_\theta^{alg \ast}$, we use all the  ${}_{-1}\alpha_1$ equations to ${}_{-1} \alpha_1$-connect the non-zero elements $(\varphi^1_{n,m}, \varphi^2_{r,s})$. We call this lattice graph as $Dgm(\varphi)$.
\par
We notice that, for a given lattice point $(n,m)$, there are three possible values at that point. They are:
\begin{enumerate}
\item $\varphi^1_{n,m}$
\item $\varphi^2_{n,m} $
\item 0.
\end{enumerate}
Hence we conclude that the kernel diagram $Dgm(\varphi)$ of $\varphi$ is a subset of  $\mathbb Z^2 \oplus \mathbb Z^2 \oplus \mathbb Z^2$. It can be easily figured out \cite{Q} that there are no edges to the graph $Dgm(\varphi)$, and the graph is a disjoint union of closed graphs with no open edges. These graphs can be infinitely supported as $\mathcal A_\theta^{alg \ast}$ consists of elements which are infinitely supported. For $1 \leq i \leq 3$, let the maps $\pi_i : \mathbb Z^2 \oplus \mathbb Z^2 \oplus \mathbb Z^2 \to \mathbb Z^2$ be the $i$-th projection, projecting the diagram $Dgm(\varphi)$ to the $i$-th $\mathbb Z^2$. From now onwards we shall deal with the map $\pi_1$ and similar agruments will hold for $\pi_2$ and $\pi_3$.

\begin{defn}[Lines]
For $s_0\in \mathbb Z$ and $\varphi(=(\varphi_1, \varphi_2)) \in ker({}_{-1}\alpha_2)$, we define a $\mathbb Z^2$ lattice  $H_{s_0}$ such that 
$$(H_{s_0})_{w,s}\coloneqq\begin{cases} 
(\pi_1(Dgm(\varphi)))_{w,s} & \text{ for } s = s_0\\
0 & else. \end{cases}$$ \\
\end{defn}

\begin{lemma} \label{thm:2.4}
Given $s_0 \in \mathbb Z$ and $\varphi \in  ker({}_{-1}\alpha_2)$, there exists $\gamma_{s_0} \in \mathcal A_\theta^{alg \ast}$ such that $\pi_1(Dgm({}_{-1}\alpha_1(\gamma_{s_0}))))_{w,s_0}=(H_{s_0})_{w,s_0}$ for all $w \in \mathbb Z$.
\end{lemma}  

\begin{proof}
We know that ${}_{-1}\alpha_1(\varphi)=(U_1^{-1}\varphi-\varphi U_1, U_2^{-1}\varphi-\varphi U_2)$. If ${}_{-1}\alpha_1(\varphi) = (\varphi_1, \varphi_2)$, then
$$\varphi^1_{n,m}=\varphi_{n+1,m} -  {\lambda}^m \varphi_{n-1,m}\text{ and } \varphi^2_{n,m}={\lambda}^{-n} \varphi_{n,m+1}-\varphi_{n,m-1}.$$
The diagram $\pi_1(Dgm({}_{-1}\alpha_1(\varphi)))$ can be infinitely supported, a connected component of it resembles the one below:
\begin{center}
\begin{tikzpicture}
\draw[step=1.0, black, thin, xshift=.5cm, yshift=.5cm](-2,-2) grid(2,2);
\fill (0,.5) circle (2pt);
\fill (0,-.5) circle (2pt);
\fill (0,1.5) circle (2pt);
\fill (0,-1.5) circle (2pt);
\fill (0,2.5) circle (2pt);
\fill (1,.5) circle (2pt);
\fill (1,-.5) circle (2pt);
\fill (1,1.5) circle (2pt);
\fill (1,-1.5) circle (2pt);
\fill (1,2.5) circle (2pt);
\fill (2,.5) circle (2pt);
\fill (2,-.5) circle (2pt);
\fill (2,1.5) circle (2pt);
\fill (2,-1.5) circle (2pt);
\fill (2,2.5) circle (2pt);
\fill (-1,.5) circle (2pt);
\fill (-1,-.5) circle (2pt);
\fill (-1,1.5) circle (2pt);
\fill (-1,-1.5) circle (2pt);
\fill (-1,2.5) circle (2pt);
\node at (.5,0) [transition]{};
\node at (-.5,0) [transition]{};
\node at (1.5,0) [transition]{};
\node at (-1.5,0) [transition]{};
\node at (2.5,0) [transition]{};
\node at (.5,1) [transition]{};
\node at (-.5,1) [transition]{};
\node at (1.5,1) [transition]{};
\node at (-1.5,1) [transition]{};
\node at (2.5,1) [transition]{};
\node at (.5,-1) [transition]{};
\node at (-.5,-1) [transition]{};
\node at (1.5,-1) [transition]{};
\node at (-1.5,-1) [transition]{};
\node at (2.5,-1) [transition]{};
\node at (.5,2) [transition]{};
\node at (-.5,2) [transition]{};
\node at (1.5,2) [transition]{};
\node at (-1.5,2) [transition]{};
\node at (2.5,2) [transition]{};
\draw[dashed](-2,0)--(4,0)node[right]{$y=s_0$.};
\draw[dashed](-1.5,2.5)--(-1.5,3)node[right]{};
\draw[dashed](-0.5,2.5)--(-0.5,3)node[right]{};
\draw[dashed](0.5,2.5)--(0.5,3)node[right]{};
\draw[dashed](1.5,2.5)--(1.5,3)node[right]{};
\draw[dashed](2.5,2.5)--(2.5,3)node[right]{};
\draw[dashed](-1.5,-2.0)--(-1.5,-1.5)node[right]{};
\draw[dashed](-0.5,-2.0)--(-0.5,-1.5)node[right]{};
\draw[dashed](0.5,-2.0)--(0.5,-1.5)node[right]{};
\draw[dashed](1.5,-2.0)--(1.5,-1.5)node[right]{};
\draw[dashed](2.5,-2.0)--(2.5,-1.5)node[right]{};
\draw[dashed](-2,1.5)--(-1.5,1.5)node[right]{};
\draw[dashed](-2,2.5)--(-1.5,2.5)node[right]{};
\draw[dashed](-2,-1.5)--(-1.5,-1.5)node[right]{};
\draw[dashed](2.5,1.5)--(3,1.5)node[right]{};
\draw[dashed](2.5,2.5)--(3,2.5)node[right]{};
\draw[dashed](2.5,-1.5)--(3,-1.5)node[right]{};
\draw[dashed](2.5,-.5)--(3,-.5)node[right]{};
\draw[dashed](2.5,.5)--(3,.5)node[right]{};
\draw[dashed](-2,-.5)--(-1.5,-.5)node[right]{};
\draw[dashed](-2,.5)--(-1.5,.5)node[right]{};
\end{tikzpicture}
\end{center}
Assume that $\varphi^1_{0,s_0} \neq 0$. It is clear from the diagram that in the row $y=s_0$ of the lattice $\pi_1(Dgm(\varphi))$, $\varphi^2_{w,s_0} =0$ for all $w \in \mathbb Z$. Define
$$(\gamma^{(1)}_{s_0})_{w,s} = \begin{cases}
-{\lambda}^{-s_0} \varphi^1_{0,s} & \text{ for } (w,s)=(-1,s_0)\\
0 & else. \end{cases}$$ \\
We have $\pi_1(Dgm({}_{-1}\alpha_1(\gamma^{(1)}_{s_0})))_{0,s_0}-(H_{s_0})_{0,s_0} = 0$. We define 
\begin{center}
$(\gamma^{(2)}_{s_0})_{w,s} = \begin{cases}
-{\lambda}^{-s_0} (\varphi^1_{-2,s}-{\lambda}^{-{s_0}} \varphi^1_{0,s})& \text{ for } (w,s)=(-3,s_0)\\
(\gamma^{(1)}_{s_0})_{w,s} & else. \end{cases}$ \\ 
\end{center}
Then we have
$$\pi_1(Dgm({}_{-1}\alpha_1(\gamma^{(2)}_{s_0})))_{-2,s_0}-(H_{s_0})_{-2,s_0} = \pi_1(Dgm({}_{-1}\alpha_1(\gamma^{(2)}_{s_0})))_{0,s_0}-(H_{s_0})_{0,s_0} = 0.$$
Similarly we can construct a  sequence $\gamma^{(n)}_{s_0}$ which satisfies the required condition for finitely many lattice points. Define $\gamma^{\leq}_{s_0} := \displaystyle \lim_{n \to \infty} \gamma^{(n)}_{s_0}$. Since $\gamma^{\leq}_{s_0} \in \mathcal A_\theta^{alg \ast}$, we have 
 $$\pi_1(Dgm({}_{-1}\alpha_1(\gamma^{\leq}_{s_0})))_{\bullet,s_0}-(H_{s_0})_{\bullet,s_0} = 0\text{ for }\bullet \leq 0.$$
We can similarly define $\gamma^{>}_{s_0}$ such that 
$$\pi_1(Dgm(({}_{-1}\alpha_1(\gamma^{>}_{s_0}))))_{\bullet,s_0}-(H_{s_0})_{\bullet,s_0} = 0\text{ for }\bullet > 0.$$
Then $\gamma_{s_0} := \gamma^{\leq}_{s_0} + \gamma^{>}_{s_0}$ satisfies the following equation:
$$\pi_1(Dgm(({}_{-1}\alpha_1(\gamma_{s_0}))))_{\bullet,s_0}-(H_{s_0})_{\bullet,s_0} = 0 \text{ for } \bullet \in \mathbb Z.$$
This completes the proof.
\end{proof}
It is interesting to note the degree of freedom that we had while constructing the above $\gamma_{s_0}$. This can be traced back to the fact that the kernel of ${}_{-1}\alpha_1$ is a $4$ dimensional vector space. As we shall prove that an arbitrary cocycle is a coboundary, it is worthwhile to note the various possibilities we have in doing so; hence revealing the nature of the map ${}_{-1}\alpha_1$.

\begin{lemma} \label{thm:hoch12}
$H^1(\mathcal A_\theta^{alg},{}_{-1}\mathcal A_\theta^{alg\ast}) = 0$.
\end{lemma}
\begin{proof}
 Let $\varphi$ belonging to $ker({}_{-1}\alpha_2)$ be a 1-cochain in the Connes complex. Let us understand the construction of $\pi_1(Dgm(\varphi))$. It consists of alternate non-zero entries, meaning one considering a row/column will find zeros at alternate positions. It has rows/columns of $\varphi^2$'s and $\varphi^1$'s alternately placed. \par
For $s_0 \in \mathbb Z$, we get a $\gamma_{s_0} \in \mathcal A_\theta^{alg \ast}$ as in Lemma  \ref{thm:2.4}. Define $\gamma = \gamma_0+ \gamma_2 + \gamma_{-2} +\cdots \in \mathcal A_\theta^{alg \ast}$. We observe that the lattice
$$\pi_1(Dgm({}_{-1}\alpha_1(\gamma)-(\varphi)))$$
has zero rows placed alternately. These rows are precisely the rows of $\varphi^1$ in $\pi_1(Dgm(\varphi))$. The other alternate set of rows is the rows of the $\varphi'^2$'s, where $\varphi'^2 \in \mathcal A_\theta^{alg \ast}$ are the bulletted points($\bullet$) in the lattice diagram $\pi_1(Dgm({}_{-1}\alpha_1(\gamma)-(\varphi)))$  . We state that $\pi_1(Dgm({}_{-1}\alpha_1(\gamma)-(\varphi)))$ is the diagram of an image element, that is, there exists $\rho \in \mathcal A_\theta^{alg \ast}$ such that 
$$\pi_1(Dgm({}_{-1}\alpha_1(\rho)) =  \pi_1(Dgm({}_{-1}\alpha_1(\gamma)-(\varphi))).$$
It is easy to see as there is no kernel equation that relates $(\varphi'_2)_{p,q}$ with $(\varphi'_2)_{l,w}$ for $ q \neq w$. Also note that if there is even a single zero entry in any of these rows, then the whole row is ought to be a \emph{zero row}. This can be seen by the repetitive application of the kernel equation to the row starting with the kernel equation containing the zero entry. 
\begin{lemma}
For $w_0 \in \mathbb Z$. There exist $\rho_{w_0} \in \mathcal A_\theta^{alg \ast}$ such that:
$$\pi_1(Dgm({}_{-1}\alpha_1(\rho)))_{w,s} := \begin{cases}
\pi_1(Dgm({}_{-1}\alpha_1(\gamma)-(\varphi)))_{w,s}& \text{ for } (w,s)=(w_0,s)\\ 
0 & else \end{cases}.$$ 
\end{lemma}
\begin{proof}
We define $\rho_{w_0}$ such that
\begin{center}
$(\rho_{w_0})_{n,m} = \begin{cases}
\varphi^1_{0,0} & \text{ for } (n,m) = (w_0,-1)\\
0  & \text{ for } (n,m)=(w_0,1) \text{ and for }(n,m) \text{ such that } n \neq w_0 . \end{cases}$\\
\end{center}
Thereafter we define $(\rho_{w_0})_{n,m}$ for $m<-1  \text{ and }n=w_0$ in the following iterated way.
$$(\rho_0)_{n,m} = - \varphi^2{'}_{n+1,m}.$$
Where $\varphi^2{'}_{n+1,m}$ is second entry of $\pi_1(Dgm({}_{-1}\alpha_1(\gamma+\rho_{w_0})-(\varphi)))^{}_{n+1,m}$.
Clearly, $\pi_1(Dgm({}_{-1}\alpha_1(\gamma+\rho_{w_0})-(\varphi)))_{w_0,s} = 0$ for all $s<0$. Similarly, we define $\rho_{w_0}$ for $m>0$ and hence we have $\rho_{w_0}$ satisfying 
$$\pi_1(Dgm({}_{-1}\alpha_1(\gamma+\rho_{w_0})-(\varphi)))_{w_0,s} = 0 \text{ for all } s \in \mathbb Z.$$
\end{proof}
Now we prove Lemma \ref{thm:hoch12}. The element $\rho =  \sum_{s \in \mathbb N}(\rho_{w_0})$ has the following property:
$$\pi_1(Dgm({}_{-1}\alpha_1(\rho)) =  \pi_1(Dgm({}_{-1}\alpha_1(\gamma)-(\varphi))).$$ 
Hence,  $H^1(\mathcal A_\theta^{alg},{}_{-1}\mathcal A_\theta^{alg \ast}) = 0$.
\end{proof}

\section{The $\mathbb Z_2$ invariant Hochschild cohomology $H^{\bullet}(\mathcal A_\theta^{alg}, \mathcal A_\theta^{alg \ast})^{\mathbb Z_2}$}

 For $g=1$, we have the following cohomology complex
$${}_{}\mathcal A_\theta^{alg \ast} \xrightarrow{{}_{}\alpha_1}{}_{}\mathcal A_\theta^{alg \ast} \oplus {}_{}\mathcal A_\theta^{alg \ast} \xrightarrow{{}_{}\alpha_2} {}_{}\mathcal A_\theta^{alg \ast} \rightarrow 0,$$
where the maps  ${}_{}\alpha_1$ and  ${}_{}\alpha_2$ are as follows:
$${}_{}\alpha_1(\varphi) = (U_1 \varphi -\varphi U_1, U_2 \varphi - \varphi U_2);$$
$${}_{}\alpha_2(\varphi_1,\varphi_2)=U_2\varphi_1 - \lambda \varphi_1 U_2 - \lambda U_1\varphi_2 + \varphi_2U_1.$$

\begin{lemma}
$H^0(\mathcal A_\theta^{alg},\mathcal A_\theta^{alg\ast})^{\mathbb Z_2} \cong \mathbb C$.
\end{lemma}

\begin{proof} Let $\varphi =\sum \varphi_{n,m}U_1^nU_2^m$ be an element of $\mathcal A_\theta^{alg \ast}$. If $\alpha_1(\varphi)=0$, then we have $ U_1\varphi -\varphi U_1=U_2\varphi - \varphi U_2=0$. This imples that we have the following relations on the coefficients: 
$$\varphi_{n-1,m} = \lambda^m \varphi_{n-1,m} = \lambda^{m+n-1} \varphi_{n-1,m}.$$
We see that these relations are satisfied only for $m=n-1=0$. Hence, we have 
$$H^0(\mathcal A_\theta^{alg},\mathcal A_\theta^{alg \ast}) \cong \mathbb C$$
and is generated by $\varphi_{0,0}$. Since the action of $\mathbb Z_2$ on the bar complex is the same as on the Connes complex, we deduce that $H^0(\mathcal A_\theta^{alg},\mathcal A_\theta^{alg\ast})^{\mathbb Z_2} \cong \mathbb C$.
\end{proof}

\begin{lemma}
$H^2(\mathcal A_\theta^{alg},\mathcal A_\theta^{alg\ast})^{\mathbb Z_2} \cong \mathbb C$.
\end{lemma}
\begin{proof}
We see from the calculations as in \cite{C} that $H^2(\mathcal A_\theta^{alg},\mathcal A_\theta^{alg\ast})^{\mathbb Z_2}= \mathcal A_\theta^{alg \ast} / Im(\alpha_2)$.
Since $\alpha_2(U_2,0) = (1-{\lambda})(U_2^2)$ and $\alpha_2(0,U_1) = (1-{\lambda})( U_1^2)$, we have $H^2(\mathcal A_\theta^{alg},\mathcal A_\theta^{alg\ast}) \cong \mathbb C$ and is generated by the cocycle equivalent to ${\varphi_{-1,-1}}$. Let $\widetilde{\varphi}_{-1,-1}$ be the corresponding element in the Connes complex.
For $\varphi \in \mathcal A_\theta^{alg \ast}$ to be $\mathbb Z_2$ invariant, we need to check that
\begin{center}
$\widetilde{\varphi}(x_0)=  \widetilde{\varphi}(k_2(I \otimes U_2^{-1}\otimes U_1^{-1}))(-1\cdot x_0)-\lambda \widetilde{\varphi}(k_2(I \otimes U_1^{-1}\otimes U_2^{-1}))(-1\cdot x_0)=\widetilde{\varphi}(U_1^{-1}U_2^{-1} \otimes U_2^{-1}U_1^{-1})(-1\cdot x_0) = \varphi(U_1^{-1}U_2^{-1}\cdot (-1 \cdot x_0)\cdot U_2^{-1}U_1^{-1})$.
\end{center}
In the above, $\widetilde{\varphi}$ is the element corresponding to $\varphi$ in the Connes complex. Considering the cocycle $\widetilde{\varphi}_{-1,-1} \in H^2(\mathcal A_\theta^{alg},\mathcal A_\theta^{alg\ast})$, we see that
$\widetilde{\varphi}_{-1,-1}(x) = x_{-1,-1}$ and $\widetilde{\varphi}(U_1^{-1}U_2^{-1}\cdot (-1 \cdot x)\cdot U_2^{-1}U_1^{-1})=x_{-1,-1}$. Hence we conclude that $\widetilde{\varphi}_{-1,-1}$ is invariant under the $\mathbb Z_2$ action.
\end{proof}

\begin{lemma}
$H^1(\mathcal A_\theta^{alg},\mathcal A_\theta^{alg\ast})^{\mathbb Z_2} =0$.
\end{lemma}

\begin{proof}
We recall from \cite{C} that:
$$H^1(\mathcal A_\theta^{alg},\mathcal A_\theta^{alg\ast})\cong \mathbb C^2$$
and is generated by ${\varphi^1_{-1,0}}$ and ${\varphi^2_{0,-1}}$. In order to locate the $\mathbb Z_2$ invariant subgroup of $H^1(\mathcal A_\theta^{alg},\mathcal A_\theta^{alg\ast})$, we use the chain homotopy maps $h_1$ and $k_1$. For $a,b \in \mathbb C$, we consider the cocycle $\chi \coloneqq (a\varphi^1_{-1,0},b\varphi^2_{0,-1}) \in \mathcal A_\theta^{alg \ast} \oplus \mathcal A_\theta^{alg \ast}$ in the Connes complex, and let  $\widetilde{\chi} \coloneqq (a\widetilde{\varphi^1}_{-1,0},b\widetilde{\varphi^2}_{0,-1}) \in \text{Hom}_{\mathcal B_\theta^{alg}}(J_1, \mathcal A_\theta^{alg})$ be the corresponding cocycle in the bar complex. It satisfyies the following relation:
$$\widetilde{\varphi^1}_{-1,0}(a \otimes b \otimes  e_1)(x)=\varphi^1_{-1,0}(bxa),  \text{ for }a,b,x \in \mathcal A_\theta^{alg \ast}.$$
Let $\psi = k_1^\ast(\widetilde{\chi})= k_1^\ast (a\widetilde{\varphi^1}_{-1,0},b\widetilde{\varphi^2}_{0,-1})=(a\widetilde{\varphi^1}_{-1,0},b\widetilde{\varphi^2}_{0,-1}) \circ k_1$ be the pushforward of $\widetilde{\chi}$. We have the following explicit description of $\psi$:
$$\psi(x_0,x_1) = (a\widetilde{\varphi^1}_{-1,0},b\widetilde{\varphi^2}_{0,-1})(k_1(I \otimes x_1))(x_0),  \text{ for }x_0,x_1 \in \mathcal A_\theta^{alg \ast}.$$
After the $\mathbb Z_2$ action $\psi$ is transformed to ${}_{-1}\psi(x_0,x_1) \coloneqq \psi(-1 \cdot x_0, -1 \cdot x_1)$. We now pullback ${}_{-1}\psi$ on to the Connes complex to compare with the cocycle $\chi$. The pullback $w\coloneqq(w_1,w_2)$ can be described as follows:
$$(w_1,w_2) = h_1^\ast ({}_{-1}\psi),\text{ where }w_i(x) \coloneqq {}_{-1}\psi(x,U_i).$$
We observe that $w_1(x) = {}_{-1}\psi(x,U_1)=\psi(-1 \cdot x,U_1^{-1})= \widetilde{a\varphi^1_{-1,0}}(k_1(I \otimes U_1^{-1}))(-1 \cdot x).$
We know from our computations \cite[Proof of Theorem 4.1]{Q} that $k_1(I \otimes U_1^{-1}) = -(U_1^{-1} \otimes U_1^{-1})$, using this we have:
\begin{center}
$ \widetilde{a\varphi^1_{-1,0}}(k_1(I \otimes U_1^{-1}))(-1 \cdot x) = -\widetilde{\varphi^1}_{-1,0}(U_1^{-1} \otimes U_1^{-1})(-1 \cdot x)= -\varphi^1_{-1,0}(U_1^{-1} \otimes U_1^{-1})(-1 \cdot x) = -\varphi^1_{-1,0}(U_1^{-1} \cdot (-1 \cdot x) \cdot U_1^{-1}) = -x_{-1,0}$.
\end{center}
Similarly, we can calculate $w_2$ and hence we finally conclude that
$$ h_1^\ast (-1 \cdot ( k_1^\ast(a\widetilde{\varphi^1}_{-1,0},b\widetilde{\varphi^2}_{0,-1}))) = -(a\widetilde{\varphi^1}_{-1,0},b\widetilde{\varphi^2}_{0,-1}).$$
Hence $\chi \notin H^1(\mathcal A_\theta^{alg}, \mathcal A_\theta^{alg \ast})^{\mathbb Z_2}$.
\end{proof}

\begin{proof} [Proof of Theorem \ref{thm:tell}] We know that the cohomology group $H^0(\mathcal A_\theta^{alg}, {}_{-1}\mathcal A_\theta^{alg\ast})^{\mathbb Z_2} \cong \mathbb C^4$ and the group $H^0(\mathcal {A}_{\theta}^{alg}, \mathcal A_\theta^{alg\ast})^{\mathbb Z_2} \cong \mathbb C$. Hence, we conclude that $H^0(\mathcal A_\theta^{alg} \rtimes \mathbb Z_2, (\mathcal A_\theta^{alg} \rtimes \mathbb Z_2)^\ast) \cong \mathbb C^{5}$.\par
We also notice that, $H^1(\mathcal A_\theta^{alg} \rtimes \mathbb Z_2, (\mathcal A_\theta^{alg} \rtimes \mathbb Z_2)^\ast)  = H^1(\mathcal A_\theta^{alg}, \mathcal A_\theta^{alg \ast})^{\mathbb Z_2} \displaystyle \oplus H^1(\mathcal A_\theta^{alg}, {}_{-1}\mathcal A_\theta^{alg \ast})^{\mathbb Z_2}\cong 0$ is clear as each of these summands is zero. As for the second Hochschild cohomology group $H^2(\mathcal A_\theta^{alg} \rtimes \mathbb Z_2, (\mathcal A_\theta^{alg} \rtimes \mathbb Z_2)^\ast)$, we observe that $H^2(\mathcal A_\theta^{alg}, \mathcal A_\theta^{alg \ast})^{\mathbb Z_2} \cong \mathbb C$ and $H^2(\mathcal A_\theta^{alg}, {}_{-1}\mathcal A_\theta^{alg \ast})^{\mathbb Z_2} =0$. Hence, we finally conclude that 
\begin{center}
 
$H^2(\mathcal A_\theta^{alg} \rtimes \mathbb Z_2, (\mathcal A_\theta^{alg} \rtimes \mathbb Z_2)^\ast)^{\mathbb Z_2} \cong \mathbb C$.
\end{center}

\end{proof}

       \section{Cyclic cohomology of $\mathcal A_\theta^{alg} \rtimes \mathbb Z_2$ }

\begin{thm} \label{thm:int}
For the algebraic noncommutative toroidal orbifold $\mathcal A_\theta^{alg} \rtimes \mathbb Z_2$, we have,
\begin{center}
$HC^0(\mathcal A_\theta^{alg} \rtimes \mathbb Z_2) \cong \mathbb C^{5}; HC^1(\mathcal A_\theta^{alg} \rtimes \mathbb Z_2) \cong 0; \newline HC^2(\mathcal A_\theta^{alg} \rtimes \mathbb Z_2) \cong \mathbb C^{6}$.
\end{center}
\end{thm}
\begin{proof} 
We consider the $S,B,I$ sequence for cohomology exact sequence.
\begin{center}
$\dots \rightarrow H^1(\mathcal A_\theta^{alg},{}_{-1}\mathcal A_\theta^{alg \ast})^{\mathbb Z_2} \xrightarrow{B} HC^0(\mathcal A_\theta^{alg},{}_{-1}\mathcal A_\theta^{alg \ast})^{\mathbb Z_2} \xrightarrow{I} HC^2(\mathcal A_\theta^{alg},{}_{-1}\mathcal A_\theta^{alg \ast})^{\mathbb Z_2} \xrightarrow{S} H^2(\mathcal A_\theta^{alg}, {}_{-1}\mathcal A_\theta^{alg \ast})^{\mathbb Z_2} \xrightarrow{B}  HC^1(\mathcal A_\theta^{alg},{}_{-1}\mathcal A_\theta^{alg \ast})^{\mathbb Z_2} \xrightarrow{I} \dots$

\end{center}
Since,  $HC^1(\mathcal A_\theta^{alg},{}_{-1}\mathcal A_\theta^{alg \ast})=H^1(\mathcal A_\theta^{alg},{}_{-1}\mathcal A_\theta^{alg \ast}) = 0$. We get $HC^2(\mathcal A_\theta^{alg},{}_{-1}\mathcal A_\theta^{alg \ast}) \cong \mathbb C^4$.
Since $H^0(\mathcal A_\theta^{alg}, {}_{-1}\mathcal A_\theta^{alg\ast})^{\mathbb Z_2} \cong \mathbb C^4$, while, $H^0(\mathcal A_\theta^{alg}, \mathcal A_\theta^{alg\ast})^{\mathbb Z_2} \cong \mathbb C$, we have,
\begin{center}
$HC^0(\mathcal A_\theta^{alg} \rtimes \mathbb Z_2) \cong \mathbb C^{5}$. 
\end{center} We see that $HC^1(\mathcal A_\theta^{alg},{}_{\pm 1}\mathcal A_\theta^{alg \ast})^{\mathbb Z_2} =0$, and hence we have 
\begin{center}
$ HC^1(\mathcal A_\theta^{alg} \rtimes \mathbb Z_2) \cong 0$.
\end{center}
 Also since $HC^2(\mathcal A_\theta^{alg},\mathcal A_\theta^{alg \ast})^{\mathbb Z_2} \cong \mathbb C^2$ and $HC^2(\mathcal A_\theta^{alg},{}_{-1}\mathcal A_\theta^{alg \ast})^{\mathbb Z_2} \cong \mathbb C^4$, we have 
$$HC^2(\mathcal A_\theta^{alg} \rtimes \mathbb Z_2) \cong \mathbb C^{6}.$$
\end{proof}
Now we can easily compute the periodic cyclic homology of the $\mathcal A_\theta^{alg} \rtimes \mathbb Z_2$.

\begin{proof} [Proof of Theorem \ref{thm:cneat}]
From the modified Connes complex we have $H^\bullet(\mathcal A_\theta^{alg} \rtimes \mathbb Z_2, (\mathcal A_\theta^{alg} \rtimes \mathbb Z_2)^\ast) = 0$ for $\bullet \geq 3$, and we have the isomorphism $HC^\bullet(\mathcal A_\theta^{alg} \rtimes \mathbb Z_2, (\mathcal A_\theta^{alg} \rtimes \mathbb Z_2)^\ast) \cong HC^{\bullet+2}(\mathcal A_\theta^{alg} \rtimes \mathbb Z_2, (\mathcal A_\theta^{alg} \rtimes \mathbb Z_2)^\ast)$ for $\bullet > 1$. Now, using the results of Theorem \ref{thm:int} we arrive at the desired results:
$$HP^{even}(\mathcal A_\theta^{alg} \rtimes \mathbb Z_2) \cong \mathbb C^6 \text{ and }HP^{odd}(\mathcal A_\theta^{alg} \rtimes \mathbb Z_2) = 0.$$
\end{proof}

\section{Chern-Connes pairing for $\mathcal A_\theta^{alg} \rtimes \mathbb Z_2$}
In this section we calculate the Chern-Connes pairing associated with the toroidal orbifold $\mathcal A_\theta^{alg} \rtimes \mathbb Z_2$. There are six projectionsgenerating $K_0(\mathcal A_\theta \rtimes \mathbb Z_2)$ \cite{ELPH}. Five of them belong to the algebra $\mathcal A_\theta^{alg} \rtimes \mathbb Z_2$ and they are the following: 
\begin{enumerate}
\item[(i)]$[1]$
\item[(ii)]$[p^{\theta}]$, where $p^{\theta}= \displaystyle \frac{1}{2}(1+t)$.
\item[(iii)]$[q_0^{\theta}]$, where $q_0^{\theta}= \displaystyle \frac{1}{2}(1-U_1t)$.
\item[(iv)]$[q_1^{\theta}]$, where $q_1^{\theta}= \displaystyle \frac{1}{2}(1-U_2t)$.
\item[(v)]$[r^{\theta}]$, where $r^{\theta}= \displaystyle \frac{1}{2}(1-{\sqrt \lambda}U_1U_2 t)$,
\end{enumerate}
where $t$ is an unitary satisfying the relations $t^2=1$ and $tU_it^{-1} = U_i^{-1}$ for $1 \leq i \leq 2$.
A complete description of the group $K_0(\mathcal A_\theta^{alg} \rtimes \mathbb Z_2)$ is unknown, with the Chern-Connes pairing of these five generators with the cyclic cocycles we will have some understanding of its noncommutative index theory. We describe pairing table for these projections with the cyclic cocyles of $HP^{even}(\mathcal A_\theta^{alg} \rtimes \mathbb Z_2)$  computed in Theorem \ref{thm:cneat}. Using the fact that $\langle [e], [S\phi] \rangle= \langle [e], [\phi] \rangle$, we have the following computations.
\begin{proof}  [Proof of Theorem \ref{thm:treat}] :\\
\underline{Pairing of $[S\tau]$} \newline
The following are the pairings with the element $[S\tau] \in HP^{even}(\mathcal A_\theta^{alg} \rtimes \mathbb Z_2)$. \\
1. $\langle [1],[\tau] \rangle = 1$ \\
2. $\langle [p^\theta],[\tau] \rangle = \displaystyle \frac{1}{2}$ \\
3. $\langle [q_0^\theta],[\tau] \rangle = \displaystyle \frac{1}{2}$ \\
4. $\langle [q_1^\theta],[\tau] \rangle = \displaystyle \frac{1}{2}$ \\
5. $\langle [r^\theta],[\tau] \rangle = \displaystyle \frac{1}{2}$. \\

\underline{Pairing of $[S\mathcal D_{0,0}]$} \newline
The following are the pairings with the element $[S\mathcal D_{0,0}] \in HP^{even}(\mathcal A_\theta^{alg} \rtimes \mathbb Z_2)$. \\
1. $\langle  [1],[\mathcal D_{0,0}] \rangle = 0$ \\
2. $\langle  [p^\theta],[\mathcal D_{0,0}] \rangle = \displaystyle \frac{1}{2}$ \\
3. $\langle  [q_0^\theta],[\mathcal D_{0,0}] \rangle = 0$ \\
4. $\langle  [q_1^\theta],[\mathcal D_{0,0}] \rangle = 0$ \\
5. $\langle  [r^\theta],[\mathcal D_{0,0}] \rangle = 0$. \\

\underline{Pairing of $[S\mathcal D_{1,0}]$} \newline
The following are the pairings with the element $[S\mathcal D_{1,0}] \in HP^{even}(\mathcal A_\theta^{alg} \rtimes \mathbb Z_2)$. \\
1. $\langle  [1],[\mathcal D_{1,0}] \rangle = 0$ \\
2. $\langle  [p^\theta],[\mathcal D_{1,0}] \rangle = 0$ \\
3. $\langle  [q_0^\theta],[\mathcal D_{1,0}] \rangle = -\displaystyle \frac{1}{2}$ \\
4. $\langle  [q_1^\theta],[\mathcal D_{1,0}] \rangle = 0$ \\
5. $\langle  [r^\theta],[\mathcal D_{1,0}] \rangle = 0$. \\

\underline{Pairing of $[S\mathcal D_{0,1}]$} \newline
The following are the pairings with the element $[S\mathcal D_{0,1}] \in HP^{even}(\mathcal A_\theta^{alg} \rtimes \mathbb Z_2)$. \\
1. $\langle  [1],[\mathcal D_{0,1}] \rangle = 0$ \\
2. $\langle  [p^\theta],[\mathcal D_{0,1}] \rangle = 0$ \\
3. $\langle  [q_0^\theta],[\mathcal D_{0,1}] \rangle = 0$ \\
4. $\langle  [q_1^\theta],[\mathcal D_{0,1}] \rangle = -\displaystyle \frac{1}{2}$ \\
5. $\langle  [r^\theta],[\mathcal D_{0,1}] \rangle = 0$ .\\

\underline{Pairing of $[S\mathcal D_{1,1}]$} \newline
The following are the pairings with the element $[S\mathcal D_{1,1}] \in HP^{even}(\mathcal A_\theta^{alg} \rtimes \mathbb Z_2)$. \\
1. $\langle  [1],[\mathcal D_{1,1}] \rangle = 0$ \\
2. $\langle  [p^\theta],[\mathcal D_{1,1}] \rangle = 0$ \\
3. $\langle  [q_0^\theta],[\mathcal D_{1,1}] \rangle = 0$ \\
4. $\langle  [q_1^\theta],[\mathcal D_{1,1}] \rangle = 0$ \\
5. $\langle  [r^\theta],[\mathcal D_{1,1}] \rangle = -\displaystyle \frac{\sqrt \lambda}{2}$.

\underline{Pairing of $[\varphi]$} \newline
The following are the pairings with the element $[\varphi] \in HP^{even}(\mathcal A_\theta^{alg} \rtimes \mathbb Z_2)$, where $\varphi$ is the even cocycle computed in the paper of A.Connes [C] \\
1. $\langle [1],[\varphi] \rangle = 0$ \\
2. $\langle [p^\theta],[\varphi] \rangle = 0$ \\
3. $\langle [q_0^\theta],[\varphi] \rangle = 0$ \\
4. $\langle [q_1^\theta],[\varphi] \rangle = 0$ \\
5. $\langle [r^\theta],[\varphi] \rangle = 0$. \\

\end{proof}

We observe that since these five projections of the algebraic noncommutative toroidal orbifold $\mathcal A_\theta^{alg} \rtimes {\mathbb Z_2}$ are projections of the smooth orbifold, $\mathcal A_\theta \rtimes \mathbb Z_2$; their linear independence in $K_0(\mathcal A_\theta \rtimes \mathbb Z_2)$ implies that they are linearly independent in $K_0(\mathcal A_\theta^{alg} \rtimes {\mathbb Z_2})$. We conjecture that these five projections span the group $K_0(\mathcal A_\theta^{alg} \rtimes {\mathbb Z_2})$.

\begin{conj}
$K_0(\mathcal A_\theta^{alg} \rtimes \mathbb Z_2) \cong \mathbb Z^5$.
\end{conj}

Acknowledgement: I acknowledge the anonymous referee(s), my postdoc mentor V. Muruganandam
and senior colleague Binod Kumar Sahoo for their writing assistance and
corrections of many mathematical typos in the earlier version of this article.

\vspace{2mm}

{\small \noindent{Safdar Quddus},\\ School of Mathematical Sciences,\\ National Institute of Science Education and Research, Bhubaneswar, India.\\
Email: safdar@niser.ac.in.


\end{document}


\begin{thebibliography}{}

\bibitem[AL]{AL} J. Alev and T. Lambre: Homologie des invariants d'une alg\`ebre de Weyl, \emph{K-Theory, 18 (1999), 401--411}.

\bibitem[B]{B} J. Baudry: Invariants du tore quantique, \emph{Bull. Sci. Math., 134 (2010), 531--547}.

\bibitem[BRT]{BRT} Y. Berest, A. Ramadoss and X. Tang: The Picard group of a noncommutative algebraic torus,  \emph{J. Noncommut. Geom., 7 (2013), 335--356.}.

\bibitem[C]{C} A. Connes: Noncommutative differential geometry, \emph{IHES Publ.
Math., 62 (1985), 257--360}.

\bibitem[ELPH]{ELPH} S. Echterhoff, W. L\"uck, N. Phillips and S. Walters: The structure of
crossed products of irrational rotation algebras by finite subgroups of $
SL_2(\mathbb{Z})$, \emph{J. Reine Angew. Math., 639 (2010), 173--221}.

\bibitem[EO]{EO} P. Etingof and A. Oblomkov: Quantization, orbifold cohomology, and Cherednik algebras, Jack, Hall-Littlewood and Macdonald polynomials, 
\emph{Contemp. Math., 417, Amer. Math. Soc., Providence, RI, (2006),  171--182.} 


\bibitem[GJ]{GJ} E. Getzler and J.D.S. Jones: The cyclic homology of crossed product
algebras, \emph{J. Reine Angew. Math., 445 (1993), 161--174}.

\bibitem[HT]{HT} G. Halbout and X. Tang: Noncommutative Poisson structures on orbifolds, \emph{Trans. Amer. Math. Soc., 362 (2010), 2249--2277}.

\bibitem[NPPT]{NPPT} N. Neumaier, M.J. Pflaum, H.B. Posthuma and X. Tang: Homology of formal deformations of proper \'etale Lie groupoids,
\emph{J. Reine Angew. Math., 593 (2006), 117--168}.


\bibitem[O]{O} A. Oblomkov: Double affine Hecke algebras of rank 1 and affine cubic surfaces, \emph{Int. Math. Res. Not., no. 18 (2004), 877--912.} 

\bibitem[PV]{PV} M. Pimsner and D. Voiculescu: Imbedding the irrational rotation
C*-algebra into an AF-algebra, \emph{J. Operator Theory, 4 (1980), 201--210}.

\bibitem[Q]{Q} S. Quddus: Hochschild and cyclic homology of the crossed product of algebraic irrational rotational algebra by finite subgraoups 
of $SL(2,\mathbb Z)$, \emph{J. Algebra 447 (2016), 322--366}.


\end{thebibliography}
\end{document}